\documentclass[reqno,11pt]{amsproc}
\usepackage{amssymb}
\usepackage{amsmath}
\usepackage{amsfonts}
\usepackage{microtype}
\usepackage[canadian]{babel}
\usepackage{xcolor}
\usepackage{geometry}
\usepackage{enumitem}
\usepackage{tikz-cd}
\usepackage{hyphenat}
\usepackage{mathtools}

\definecolor{myurlcolor}{rgb}{0,0,0.4}
\definecolor{mycitecolor}{rgb}{0,0.5,0}
\definecolor{myrefcolor}{rgb}{0.5,0,0}
\usepackage[pagebackref,draft=false]{hyperref}
\hypersetup{colorlinks,
linkcolor=myrefcolor,
citecolor=mycitecolor,
urlcolor=myurlcolor}

\usepackage[capitalize]{cleveref}

\newcommand{\beq}{\begin{equation}}
\newcommand{\eeq}{\end{equation}}
\newcommand{\N}{\mathbb{N}}
\newcommand{\Nplus}{\mathbb{N}_{> 0}}			% *strictly* positive natural numbers
\newcommand{\Z}{\mathbb{Z}}
\newcommand{\Q}{\mathbb{Q}}
\newcommand{\R}{\mathbb{R}}
\newcommand{\TR}{\mathbb{TR}}
\newcommand{\Rplus}{\mathbb{R}_{> 0}}			% *strictly* positive natural numbers
\newcommand{\C}{\mathbb{C}}
\newcommand{\op}{\mathrm{op}}
\newcommand{\eps}{\varepsilon}

\newcommand{\supp}[1]{\mathrm{supp}(#1)}		% support of a measure

\DeclareMathOperator{\down}{\downarrow\!\,}		% generated downset
\DeclareMathOperator{\up}{\uparrow\!\,}		% generated downset
\newcommand{\Sper}[1]{\mathsf{TSper}(#1)}		% real spectrum
\newcommand{\Proj}[1]{\mathbb{P}(#1)}			% projective space of positive cone
\newcommand{\lev}{\mathrm{lev}}			% logarithmic evaluation map
\newcommand{\lc}{\mathrm{lc}}			% logarithmic comparison map
\DeclareMathOperator*\uplim{\overline{lim}}	% overlined lim

\newcommand{\cat}[1]{\mathsf{#1}}	% general font for semirings
\newcommand{\Set}{\cat{Set}}
\newcommand{\Haus}{\cat{Haus}}

\newcommand{\meas}[1]{\mathcal{M}{(#1)}}	% Radon measures
\newcommand{\measc}[1]{\mathcal{M}_c{(#1)}}	% compactly supported Radon measures

\usepackage{enumitem}
\setlist[enumerate]{label=(\alph*),itemsep=5pt,topsep=6pt}
\setlist[itemize]{label=$\triangleright$,itemsep=5pt,topsep=6pt}

\Crefformat{enumi}{#2#1#3}

\swapnumbers
\newtheorem{dummy}{Dummy}[section]
\newtheorem{thm}[dummy]{Theorem}\Crefname{thm}{Theorem}{Theorems}
\newtheorem{lem}[dummy]{Lemma}\Crefname{lem}{Lemma}{Lemmas}
\newtheorem{prop}[dummy]{Proposition}\Crefname{prop}{Proposition}{Propositions}
\newtheorem{cor}[dummy]{Corollary}\Crefname{cor}{Corollary}{Corollaries}
\Crefname{conj}{Conjecture}{Conjectures}
\Crefname{qstn}{Question}{Questions}
\newtheorem{defn}[dummy]{Definition}\Crefname{defn}{Definition}{Definitions}
\Crefname{prob}{Problem}{Problems}
\Crefname{nota}{Notation}{Notations}
\theoremstyle{remark}
\newtheorem{ex}[dummy]{Example}\Crefname{ex}{Example}{Examples}
\newtheorem{rem}[dummy]{Remark}\Crefname{rem}{Remark}{Remarks}
\Crefname{note}{Note}{Notes}
\numberwithin{equation}{section}
\Crefname{enumi}{}{}

\setlist[enumerate,1]{label={(\alph*)}}
\setlist[enumerate,2]{label={(\roman*)},ref={(\roman*)}}
\Crefformat{enumi}{#2#1#3}

\allowdisplaybreaks

\protected\def\verythinspace{%
	\ifmmode
		\mskip0.5\thinmuskip
	\else
		\ifhmode
			\kern0.08334em
		\fi
	\fi}
\renewcommand{\,}{\verythinspace}

\let\originalleft\left
\let\originalright\right
\renewcommand{\left}{\mathopen{}\mathclose\bgroup\originalleft}
\renewcommand{\right}{\aftergroup\egroup\originalright}

\newcommand{\newterm}[1]{\textbf{#1}}

\usepackage{todonotes}

\begin{document}

\title[Asymptotic and catalytic stochastic orders]{Characterizing the asymptotic and catalytic\\ stochastic orders on topological abelian groups}

\author{\smallskip Tobias Fritz}

\address{Department of Mathematics, University of Innsbruck, Austria}
\email{tobias.fritz@uibk.ac.at}



\thanks{\textit{Acknowledgements.} We thank Richard K\"ung, Rostislav Matveev, Luciano Pomatto, Matteo Smerlak, Arleta Szko{\l}a, Omer Tamuz and P\'eter Vrana for useful discussions and feedback, as well as David Handelman and Terence Tao for discussion on MathOverflow. Part of this work has been conducted while the author was with the Max Planck Institute for Mathematics in the Sciences and later with the Perimeter Institute for Theoretical Physics, both of which we thank for their outstanding research environments.}

\begin{abstract}
	We study the usual stochastic order between probability measures on preordered topological abelian groups, focusing on asymptotic and catalytic versions of the order.
	In the asymptotic version, a measure $\mu$ dominates a measure $\nu$ if the i.i.d.~random walk generated by $\mu$ first-order dominates the one generated by $\nu$ at late times.
	In the catalytic version, $\mu$ dominates $\nu$ if there is a third $\tau$ such that the convolution $\mu \ast \tau$ first-order dominates $\nu \ast \tau$.

	Provided that the preorder on $G$ is induced by a suitably large positive cone and that both measures are compactly supported Radon, our main result gives a sufficient condition for asymptotic and catalytic dominance to hold in terms of a family of inequalities closely related to the cumulant-generating functions.
	While this sufficient condition requires these inequalities to be strict, the non-strict versions of these inequalities are easily seen to be necessary.
	In this sense, our result gives conditions that are necessary and sufficient in generic cases.
	This result has been known for $G = \R$, but is new already for $\R^n$ with $n > 1$.
	It is a direct application of a recently proven theorem of real algebra, namely a \emph{Vergleichsstellensatz} for preordered semirings. 

	We finally use our result to derive a formula for the rate at which the probabilities of a random walk decay \emph{relative} to those of another, now for walks on a preordered topological vector space with compactly supported Radon steps. Taking one of these walks to be deterministic reproduces a version of Cram\'er's large deviation theorem for infinite dimensions.
\end{abstract}

\newcommand{\E}[1]{\mathbb{E}[#1]}
\renewcommand{\P}[1]{\mathbf{P}\left[{#1}\right]}

%%% title page
\newgeometry{top=2cm,bottom=3cm}
\maketitle
\thispagestyle{empty}
\tableofcontents
\restoregeometry

\section{Introduction}

Probability theory offers many classical results on the asymptotic behaviour of random walks, including the strong and weak laws of large numbers, the central limit theorem, and Cram\'er's large deviation theorem. In this paper, we are concerned with the \emph{comparison} of two random walks, and in particular on when random walk dominates another one at late times in \emph{first-order stochastic dominance}, also known as the \emph{usual stochastic order}.
A closely related problem turns out to be that of \emph{catalytic} stochastic dominance, where the comparison is made non-asymptotically but after adding a third independent random variable to both of the original ones.

Here is the result of Aubrun and Nechita on this asymptotic and catalytic dominance, which we will generalize in this paper.

\begin{thm}[{\cite{AN}}]
	\label{Rintro_thm}
	Let random variables $X$ and $Y$ be real-valued and bounded, and let $(X_i)_{i\in\N}$ and $(Y_i)_{i\in\N}$ be i.i.d.~copies. Consider the following conditions:
	\begin{enumerate}[label=(\roman*)]
		\item\label{Rintro_catalytic} There is a random variable $Z$, independent of $X$ and $Y$, such that
			\beq
				\label{catalytic_ineq}
				\P{X + Z \ge c} \le \P{Y + Z \ge c} \quad \forall c \in \R.
			\eeq
		\item\label{Rintro_asymptotic} For $n \ge 1$,
			\beq
				\P{\sum_{i=1}^n X_i \ge c} \le \P{\sum_{i=1}^n Y_i \ge c} \quad \forall c \in \R.
			\eeq
		\item\label{Rintro_ineqs} With $\prec$ standing for $<$ or $\le$, the following hold:
			\begin{align}
				\E{e^{tX}} & \prec \E{e^{tY}},		& \E{e^{-tX}} & \succ \E{e^{-tY}} \qquad \forall t \in \Rplus \nonumber \\[6pt]
				\max X & \prec \max Y,			& \min X & \prec \min Y, \label{Rintro_ineqs_eq}
			\end{align}
			\[
				\E{X} \prec \E{Y}.
			\]
	\end{enumerate}
	Then \ref{Rintro_catalytic} or \ref{Rintro_asymptotic} for some $n \ge 1$ implies that \ref{Rintro_ineqs} holds with non-strict inequalities. Conversely if \ref{Rintro_ineqs} holds with strict inequalities, then \ref{Rintro_catalytic} and \ref{Rintro_asymptotic} for all $n \gg 1$ follow.
\end{thm}

For example, the stochastic order plays an important role in decision theory and economic theory~\cite{microecon}.
In that context, it is certainly of interest to consider its asymptotic and catalytic versions, with the asymptotic version corresponding e.g.~to an idealized simplification of comparing the growth of two portfolios at large times, and the catalytic version corresponding to a comparison of two portfolios when each is combined with a third independent one.

The main result of this paper extends \Cref{Rintro_thm} to the higher-dimensional case, and in fact to random variables with values in arbitrary topological abelian groups $G$, where the stochastic order is defined with respect to any suitably large positive cone $G_+ \subseteq G$. This is new already for the case of $\R^n$ with $n > 1$. We will state and prove this as \Cref{main_thm}. 
Both in \Cref{Rintro_thm} and in our \Cref{main_thm}, the implication from \ref{Rintro_catalytic} or \ref{Rintro_asymptotic} to \ref{Rintro_ineqs} with non-strict inequalities is easy to see and follows simply by the fact that the relevant quantities are monotone with respect to stochastic order and are multiplicative (resp.~additive) with respect to sums of independent variables.
The difficult direction is the converse implication formulated in the final sentence.
While this was proven by Aubrun and Nechita for \Cref{Rintro_thm} by conventional large deviation methods, these methods do not seem to apply in the higher-dimensional case. Instead, our proof of the more general \Cref{main_thm} proceeds by showing that it is an instance of our recent \emph{Vergleichsstellensatz}\footnote{This terminology is by analogy with the \emph{Nullstellensatz} from algebraic geometry, which is conceptually similar result.} for preordered semirings~\cite[Theorem~8.6]{vssII}. In other words, our proof is by reduction to a purely algebraic result.

We imagine that our \Cref{main_thm} may have a similar significance for decision theory and economics as \Cref{Rintro_thm} does.
For example, consider the problem of comparing two portfolios denominated in two different currencies.
Since the exchange rate between the two currencies is itself subject to uncertainty, one may not want to assume any particular exchange rate.
However, it is still possible to compare the two portfolios by applying \Cref{main_thm} with $G = \R^2$ and $G_+ = \R_+^2$.
In this way, one would prefer one portfolio over the other if there is a joint distribution of returns such that the first portfolio's return is almost surely as much as the second's (\Cref{stochastic_order_char}), and our \Cref{main_thm} characterizes when this holds asymptotically or catalytically.

Finally, in \Cref{uniformd_ld} we derive a formula for how the tail probabilities of one random walk decay relative to those of another random walk. This is a consequence of \Cref{main_thm}, and as such applies to random walks with compactly supported Radon steps on topological vector spaces equipped with a suitably large positive cone. In the $\R$-valued case, this takes the following slightly simplified form.

\begin{thm}
	\label{Rintro_uniform_ld}
	For bounded real-valued random variables $X$ and $Y$ with i.i.d.~copies $(X_i)_{i \in \N}$ and $(Y_i)_{i \in \N}$, we have
	\beq
		\label{Rintro_quotient_mgf}
		\sup_{\eps > 0} \, \uplim_{n \to \infty} \, \sup_{c \in \R} \, \frac{1}{n} \log \frac{\P{\frac{1}{n}\sum_{i=1}^n X_i \ge c}}{\P{\frac{1}{n}\sum_{i=1}^n Y_i \ge c - \eps}} = \sup_{t \ge 0} \log \frac{\E{e^{t X}}}{\E{e^{t Y}}},
	\eeq
	where this equation holds in two versions, with $\uplim_{n \to \infty}$ standing for $\liminf_{n \to \infty}$ or $\limsup_{n \to \infty}$.
\end{thm}

Taking $X$ to be deterministic recovers a version of Cram\'er's large deviation theorem (for bounded $Y$). We show this in \Cref{cramer} for random variables taking values in topological vector spaces. In the proof of that result, it is instructive to see how the rate function, being the Legendre--Fenchel transform of the cumulant-generating function of $Y$, arises from the right-hand side of \eqref{Rintro_quotient_mgf} upon taking $X$ to be deterministic.

\subsection*{Summary}

We now briefly summarize the content of the individual sections of this paper.

\begin{itemize}
	\item We recall some measure-theoretic preliminaries in \Cref{prelims}, relevant in particular to the infinite-dimensional case.
	\item We discuss preordered topological abelian groups in \Cref{potag}, including the introduction of an order unit condition that appears in our main results (\Cref{nicely_ordered}).
	\item We consider preordered semialgebras in \Cref{meas_semialg} and restate a simplified version of our \emph{Vergleichsstellensatz}~\cite[Theorem~8.6]{vssII} as \Cref{vss}. This is the algebraic result from which our present results will follow as an instance, and which has also found applications in representation theory~\cite{rep_app} and asymptotic statistics~\cite{major}.
		To better organize the structure of the inequalities which appear in the \emph{Vergleichsstellensatz}, we recall the \emph{test spectrum} of a preordered semialgebra in \Cref{tsper_defn}.

		Moreover, we also explain how compactly supported Radon measures on a preordered topological abelian group (with an order unit) form a preordered semialgebra satisfying the relevant assumptions, with multiplication given by the convolution of measures.
	\item \Cref{main_sec} then states and proves our main result as \Cref{main_thm}. We also switch from measure-theoretic terminology and notation to the more intuitive random variables language.
	\item \Cref{ncgf} investigates the test spectrum for the preordered semialgebra of measures under convolution further.
		We explain how this test spectrum recovers the normalized cumulant-generating function, and how this function rewrites all of the inequalities in \eqref{Rintro_ineqs_eq} in a unified form.
	\item Finally, \Cref{sec_uld} uses our results of the previous two sections to derive the general version of \Cref{Rintro_uniform_ld} as \Cref{uniformd_ld}.
		To illustrate that this result is nontrivial even when one of the variables is deterministic, we derive an infinite-dimensional version of Cram\'er's large deviation theorem as \Cref{cramer}.
\end{itemize}

\section{Radon measures on Hausdorff spaces}
\label{prelims}

We start with some measure-theoretic preliminaries, which are relevant mainly for getting maximum mileage out of our methods by allowing us to treat the general case of topological abelian groups rather than e.g.~merely $\R^n$. Readers who are only interested in the finite-dimensional situation of $\R^n$, where our results are still new (for $n > 1$) and nontrivial, can safely skip this preliminary section.

We write $\Haus$ for the category of Hausdorff spaces and continuous maps. For $A \in \Haus$, we denote by $\meas{A}$ the set of finite (unsigned) \newterm{Radon measures} on $A$, i.e.~the set of finite Borel measures that are inner regular.
The regularity of $\mu \in \meas{A}$ guarantees that there is a largest closed set of full measure, called the \newterm{support} $\supp{\mu}$.
If $f : A \to B$ is continuous, then pushforward of measures defines a map $\meas{f} : \meas{A} \to \meas{B}$~\cite[Section~I.5]{schwartz}. We thereby obtain a functor $\mathcal{M} : \Haus \to \Set$.

For $A, B \in \Haus$, their product space $A \times B$ is again in $\Haus$. Applying the functoriality to the product projections $A \times B \to A$ and $A \times B \to B$ produces the two components of the \newterm{marginalization map}
\beq
	\label{marginals}
	\Delta_{A,B} \: : \: \meas{A \times B} \longrightarrow \meas{A} \times \meas{B}.
\eeq
In the other direction, the formation of product measures~\cite[p.~63]{schwartz} induces a \newterm{product map}
\beq
	\label{products}
	\nabla_{A,B} \: : \: \meas{A} \times \meas{B} \longrightarrow \meas{A \times B},
\eeq
where the support of the product measure on $A \times B$ is exactly the product of the respective supports in $A$ and $B$. Both $\Delta_{A,B}$ and $\nabla_{A,B}$ are natural in $A$ and $B$ in the sense of category theory. We refer to~\cite{bimonoidal} for a general theory of these maps and the equations they satisfy.

The \newterm{compactly supported} finite Radon measures form a subset $\measc{A} \subseteq \meas{A}$.
It is easy to see that the marginalization and product maps above restrict to corresponding maps on $\mathcal{M}_c$, namely for any $A, B \in \Haus$, we have
\begin{alignat}{2}
	\label{marginalsc}
	& \textrm{Marginalization map:} \quad\: & & \Delta_{A,B} \: : \: \measc{A \times B} \longrightarrow \measc{A} \times \measc{B}, \\
	\label{productsc}
	& \textrm{Product map:} \quad\: & & \nabla_{A,B} \: : \: \measc{A} \times \measc{B} \longrightarrow \measc{A \times B}.
\end{alignat}

\section{Preordered topological abelian groups}
\label{potag}

In this section, we state and discuss the relevant definitions concerning preordered topological abelian groups. For us, a \newterm{topological group} is a group $G$ with a Hausdorff topology such that both the multiplication map $G \times G \to G$ and the inversion map $G \to G$ are continuous.
Throughout, we work with topological abelian groups using additive notation.

\begin{defn}
A topological abelian group $G$ is \newterm{preordered} if it comes equipped with a \newterm{positive cone}, which is a distinguished closed subset $G_+ \subseteq G$ with
\[
	G_+ + G_+ \subseteq G_+, \qquad 0 \in G_+.
\]
\end{defn}

We refer to~\cite{goodearl} for further background on preordered abelian groups in the purely algebraic context. 

The paradigmatic examples that we have in mind at this point are $G = \R^d$ with $G_+$ any closed convex cone, or more generally any topological vector space equipped with a closed convex cone~\cite{AT}. In the latter case, we will say that $G$ is a \newterm{preordered topological vector space}.
However, if $G$ is a topological vector space, then the positive cone $G_+$ is not automatically closed under positive scalar multiplication: taking $G = \R$ and
\[
	G_+ \coloneqq \{0\} \cup [1,\infty)
\]
still produces a preordered topological abelian group in our sense, but not a preordered topological vector space.
The following example is even more peculiar.

\begin{rem}
	For a preordered topological abelian group $G$, the set $G_+ - G_+$ is automatically a subgroup of $G$, but it need not be open or closed.

	For example, consider the topological abelian group $G = \R$ with positive cone
	$G_+$ given by zero together with all rationals $\frac{p}{q}$ satisfying $q > 0$ and $\frac{p}{q} \ge \log q$.
	A short computation shows that this set is indeed closed under addition.
	It is topologically closed since it contains only finitely many points in every bounded interval.
	However, we clearly have $G_+ - G_+ = \Q$.
\end{rem}

For $a, b \in G$, we write $a \le b$ as usual if $b - a \in G_+$, defining a preorder relation on $G$ which is translation-invariant. In terms of this, we also have the \newterm{order interval}
\[
	[a,b] \coloneqq \{ x\in G \mid a \le x \le b \} = (a + G_+) \cap (b - G_+),
\]
which is clearly closed.
For a subset $S \subseteq G$, we also write
\[
	\down{S} \coloneqq \{x \in G \mid \exists s \in S, \: x \le s \} = S - G_+
\]
for the \newterm{downset} generated by $S$, and similarly $\up{S}$ for the \newterm{upset} $\up{S} \coloneqq S + G_+$. A set $S$ is \newterm{downward closed} if it is equal to its own downset; and similarly $S$ is \newterm{upwards closed} if it is its own upset.

The following definition is standard, at least in the purely algebraic setting~\cite[p.~4]{goodearl}.

\begin{defn}
	\label{nicely_ordered}
	Let $G$ be a preordered topological abelian group. Then an \newterm{order unit} is an element $u \in G_+$ such that:
	\begin{enumerate}
		\item\label{orderunit} For every $x \in G$ there is $k \in \N$ with $x \le ku$.
		\item\label{intervalnbhd} The order interval $[-u,+u]$ is a neighbourhood of $0 \in G$.
	\end{enumerate}
\end{defn}

In particular, if $G_+$ has an order unit, then $G = G_+ - G_+$.

\begin{rem}
	We comment on the relation between these two conditions. For preordered topological vector spaces, it is well-known that \ref{intervalnbhd} implies \ref{orderunit}~\cite[Lemma~2.5]{AT}. But this is not true for preordered topological abelian groups in general. An almost trivial example is $G = \Z$ and $G_+ = \{0\}$, for which $u = 0$ satisfies \ref{intervalnbhd} but not \ref{orderunit}.
	
	The other direction already fails for preordered topological vector spaces. For example, consider $C([0,1])$ equipped with the weak-$*$ topology and preordered with respect to the usual closed convex cone containing the nonnegative functions. Then the constant function $u \coloneqq 1$ satisfies~\ref{orderunit}, but the order interval $[-1,+1]$ in $C([0,1])$ is not a neighbourhood of zero.
\end{rem}

\begin{ex}
	If $G$ is finite, then every submonoid $G_+ \subseteq G$ is a positive cone. Since every element is torsion, the positive cones are then exactly the subgroups. It follows that $G_+$ has an order unit if and only if $G = G_+$, in which case every element is an order unit.
\end{ex}

\begin{ex}
	For $G = \Z^d$, a positive cone $G_+$ has an order unit if and only if $G = G_+ - G_+$. Indeed if this condition holds, then we can write the standard basis vectors as $e_i = x_i - y_i$ for $x_i,y_i \in G_+$ for all $i=1,\ldots,d$. Hence $u \coloneqq x_1 + \ldots + x_d$ is an order unit.
\end{ex}

\begin{ex}
	For $G = \R^d$ as a topological vector space, consider any closed convex cone $G_+ \subseteq \R^d$. Then an element of $G_+$ is an order unit if and only if it is a topologically interior point.
\end{ex}

\begin{lem}
	\label{compact_bounded}
	Let $u \in G_+$ be an order unit. Then for every compact $C \subseteq G$ there is $k \in \N$ with
	\[
		C \subseteq \down \{ku\} \cap \up\{-ku\}.
	\]
\end{lem}

\begin{proof}
	It is enough to prove $C \subseteq \down\{ku\}$ for some $k$, since then $C \subseteq \up\{-ku\}$ for some $k$ follows by symmetry. For every $x \in C$ we have $k_x \in \N$ with $x \le k_x u$. But then also $x' \le (k_x + 1) u$ for every $x' \in [x - u, x + u]$. Since this order interval is a neighbourhood of $x$, the compactness implies that there are finitely many $x_1,\ldots,x_n \in C$ such that $C \subseteq \bigcup_i [x_i - u, x_i + u]$. With $k \coloneqq \max_{i=1,\ldots,n} k_{x_i}$, the claim $C \subseteq \down \{k u\}$ now follows.
\end{proof}

\section{The preordered semialgebra of measures}
\label{meas_semialg}

\subsection*{Convolution}

If $G$ is a topological abelian group, then the multiplication $G \times G \to G$ induces the \newterm{convolution of measures} map defined as the composition
\beq
	\label{convolute}
	\meas{G} \times \meas{G} \longrightarrow \meas{G \times G} \longrightarrow \meas{G},
\eeq
where the first map is an instance of~\eqref{products} and the second one is by functoriality of $\mathcal{M}$ applied to the multiplication map.
More explicitly, the convolution can be characterized in terms of how to integrate against it: for a bounded measurable function $f : G \to \R$ and two measures $\mu,\nu \in \meas{G}$, we have~\cite[\S{}444]{fremlin},
\beq
	\label{integral_convolute}
	\int f \, d(\mu \ast \nu) = \iint f(x + y) \, d\mu(x) \, d\nu(y).
\eeq
This convolution operation turns $\meas{G}$ into a commutative monoid with neutral element $\delta_0$. A convenient way of proving the relevant associativity and commutativity properties is to use the corresponding associativity and commutativity properties of the formation of product measures~\eqref{products}, which amount to the fact that $\mathcal{M}$ is a lax symmetric monoidal functor. The traditional computational proof using the explicit formula for convolution~\cite[Section~2.5]{folland} is the same in spirit.

By \eqref{productsc}, it follows that the set of compactly supported Radon measures $\measc{G}$ is closed under convolution, and therefore becomes a submonoid of $\meas{G}$.
For $x,y \in G$, we have $\delta_x \ast \delta_y = \delta_{x + y}$, and this makes the inclusion
\beq
	\label{unit_monad}
	G \longrightarrow \measc{G}, \qquad x \longmapsto \delta_x
\eeq
into a homomorphism of commutative monoids.

Recall also that if $X$ and $Y$ are independent $G$-valued random variables with distributions $\mu$ and $\nu$, then $\mu \ast \nu$ is the distribution of the $G$-valued variable $X + Y$.

\subsection*{The stochastic preorder}

Suppose now that $G$ is a preordered topological abelian group, where the preorder is defined through a positive cone $G_+$. We now extend the resulting preorder on $G$ to a preorder on $\meas{G}$ known as the stochastic preorder. We refer to Strassen~\cite[Theorem~11]{strassen}, Edwards~\cite[Theorem~7.1]{edwards} and Kellerer~\cite[Proposition~3.12]{kellerer} for more general definitions and proofs of the following equivalence, which crucially rely on the assumption that the measures involved are Radon, but neither use the group structure nor Hausdorffness of $G$.

\begin{prop}
	\label{stochastic_order_char}
	For $\mu,\nu \in \meas{G}$ with $\mu(G) = \nu(G)$, the following are equivalent:
	\begin{enumerate}
		\item $\mu(C) \le \nu(C)$ for every closed upset $C \subseteq G$.
		\item $\mu(U) \le \mu(U)$ for every open upset $U \subseteq G$.
		\item For every monotone and lower semi-continuous function $f : G \to \R$, we have
			\begin{equation}
				\label{stochastic_order_integral}
				\int f \, d\mu \le \int f \, d\nu.
			\end{equation}
		\item\label{joint} There is $\lambda \in \meas{G \times G}$ with marginals $\mu$ and $\nu$, and such that $\lambda$ is supported on the preorder relation $\{(x,y) \mid x \le y\} \subseteq G \times G$.
	\end{enumerate}
\end{prop}

Throughout the rest of the paper, we call a map $f$ ``monotone'' if it is order-preserving, that is $x \le y$ implies $f(x) \le f(y)$.

\begin{defn}
	The \newterm{stochastic preorder} is the relation on $\meas{G}$ defined by these equivalent conditions.
\end{defn}

We also write $\mu \le \nu$ to denote this relation for $\mu,\nu \in \meas{G}$. By definition, $\mu \le \nu$ can hold only if the normalizations are the same, $\mu(G) = \nu(G)$. Intuitively, $\mu \le \nu$ means that $\nu$ can be obtained from $\mu$ by merely moving mass upwards in the preorder. Note that $x \le y$ in $G$ is equivalent to $\delta_x \le \delta_y$ in $\meas{G}$.

Either of the first three equivalent conditions obviously shows that $\le$ is a preorder relation, i.e.~is reflexive and transitive. If the preorder on $G$ is antisymmetric, or equivalently if $G_+ \cap (-G_+) = \{0\}$, then it is known that the stochastic preorder is antisymmetric too~\cite{antisymmetry}. The most well-known instance of the stochastic preorder is for $G = \R$ and $G_+ = \R_+$, in which case it is also called the \newterm{usual stochastic order} or \newterm{first-order stochastic dominance}. In this case, we have $\mu \le \nu$ if and only if
\beq
	\label{R_stoch_order}
	\mu( [c,\infty) ) \le \nu( [c,\infty) ) \qquad \forall c\in\R,
\eeq
since in this case the closed and upward closed sets are exactly the $[c,\infty)$. This system of inequalities can be understand intuitively upon thinking of $\mu$ and $\nu$ as return distributions of a financial asset: then this condition states that the return distribution described by $\nu$ is unambiguously (non-strictly) preferable over the one given by $\mu$~\cite{HR}.

For later use, we record a simple observation relating the stochastic preorder with supports.

\begin{lem}
	\label{supp_crit}
	Suppose that $\mu,\nu \in \meas{G}$ with $\mu(G) = \nu(G)$ are such that $x \le y$ for all $x \in \supp{\mu}$ and $y \in \supp{\nu}$. Then $\mu \le \nu$.
\end{lem}

\begin{proof}
	We assume $\mu(G) = \nu(G) = 1$ without loss of generality. Then this follows e.g.~from condition~\ref{joint} of \Cref{stochastic_order_char} upon taking $\lambda$ to be the product measure $\mu \otimes \nu$.
\end{proof}

\subsection*{Preordered semialgebra structure}

$\meas{G}$ carries both an additive commutative monoid structure given by addition of measures, as well as the commutative monoid structure given by convolution~\eqref{convolute}, which distributes over the addition. This makes $\meas{G}$ into a \newterm{commutative semiring}.
While we refer to the literature for the full definitions~\cite{golan}, it may help to note that a semiring is like a ring, except in that additive inverses generally do not exist. 
As usual, we denote its operations by $+$ and $\cdot$ and the corresponding neutral elements by $0$ and $1$. 
Since all of the semirings considered in this paper are commutative, we no longer mention the commutativity assumption explicitly.

Thus $\measc{G}$ is a semiring and $\meas{G}$ is a subsemiring.
Both semirings also carry an additional scalar multiplication by nonnegative reals, as per the following definition.

\begin{defn}
	A \newterm{semialgebra} $S$ is a semiring together with a map
	\[
		\R_+ \times S \longrightarrow S
	\]
	which is additive\footnote{By additivity, we mean both that binary addition and nullary addition, i.e.~the neutral element $0$, are preserved.} in each argument and satisfies $1x = x$ as well as $(rx)(sy) = (rs)(xy)$ for all $r,s \in \R_+$ and $x,y \in S$.
\end{defn}

Since we will not consider scalar multiplication by any other semiring than $\R_+$, we leave out mention of $\R_+$ in the term ``semialgebra''.
Here are the two most important examples that are relevant for the theory of preordered semirings itself~\cite{vssII}.

\begin{ex}
	\label{rtr}
	\begin{enumerate}
		\item $\R_+$ itself, with its usual algebraic structure, is an $\R_+$-semialgebra.
		\item The \newterm{tropical reals} $\TR_+$ are the $\R_+$-semialgebra given by the semiring\footnote{The tropical reals are usually defined as a different but isomorphic semiring, namely $(\R\cup\{-\infty\}, \max, +)$, where the logarithm and exponential implement an isomorphism between this definition and ours. The multiplicative version that we use turns out to be more convenient for our purposes, in particular \Cref{tsper_defn}.}
			\[
				(\R_+, \max, \cdot),
			\]
			meaning that addition is formation of the maximum with neutral element $0$, while multiplication is as usual. As for scalar multiplication, we put $rx \coloneqq x$ for all $r \in \Rplus$ as well as $0x \coloneqq 0$. 
	\end{enumerate}
\end{ex}

Our main object of study will be $\measc{G}$, considered as a semialgebra and together with the stochastic preorder. What compatibility is there between the algebraic structure and the preorder?
It is straightforward to see that the following compatibility holds on $\measc{G}$.

\begin{defn}
	A \newterm{preordered semiring} $S$ is a semiring together with a preorder relation $\le$ such that for all $a,x,y \in S$,
	\[
		x \le y \qquad \Longrightarrow \qquad a + x \le a + y, \qquad ax \le ay.	
	\]
	A \newterm{preordered semialgebra} is a semialgebra which is preordered as a semiring.
\end{defn}

Note that the scalar multiplication $x \mapsto rx$ for $r \in \R_+$ is automatically monotone, since $x \le y$ implies that $rx = (r1)x \le (r1)y = ry$. We present some examples, starting with two generically important ones and then our main object of study.

\begin{ex}
	If $S$ is a preordered semialgebra, then we write $S^\op$ for the same semialgebra but with the \newterm{opposite preorder}, meaning that $x \le y$ holds in $S^\op$ if and only if $x \ge y$ holds in $S$.
	Clearly $S^\op$ is again a preordered semialgebra.
\end{ex}

\begin{ex}
	Consider the semialgebras $\R_+$ and $\TR_+$ from \Cref{rtr}. Both of these are preordered semialgebras with respect to the usual order on the real numbers. $\R_+^\op$ and $\TR_+^\op$ are the same preordered semialgebras carrying the opposite of the usual order.
\end{ex}

\begin{ex}
	Let $G$ be a preordered topological abelian group. Then $\measc{G}$ is a preordered semialgebra with respect to addition, scalar multiplication and convolution of measures as algebraic operations, and with respect to the stochastic preorder as preorder.
	Indeed the monotonicity of addition is an obvious consequence of either condition in \Cref{stochastic_order_char}, while the monotonicity of multiplication is perhaps most easily seen from \eqref{stochastic_order_integral}: if $\mu \le \nu$ and $\tau \in \measc{G}$, then for every monotone and lower semi-continuous $f$, we have
	\[
		\int f \, d(\tau \ast \mu) = \iint f(x + y) \, d\tau(x) \, d\mu(y) \le \iint f(x + y) \, d\tau(x) \, d\nu(y) = \int f \, d(\tau \ast \nu), 
	\]
	where the inequality holds by the assumed $\mu \le \nu$ and because $y \mapsto \int f(x + y) \, d\tau(x)$ is still a monotone and lower semi-continuous function on $G$.
\end{ex}

The following growth condition plays a key role in our theory of preordered semirings.

\begin{defn}[{\cite[Definition~3.28]{vssI}}]
\label{univ_defn}
Let $S$ be a preordered semiring. An element $v\in S$ with $v\geq 1$ is \newterm{power universal} if for every $x \le y$ in $S$, there is $k \in \N$ such that
\beq
\label{power_univ}
	y \le v^k x.
\eeq
\end{defn}

The following shows that this indeed applies in the case of interest to us.
We continue assuming that $G$ is a preordered topological abelian group and prove two auxiliary statements about the preordered semialgebra $\measc{G}$.

\begin{lem}
	\label{MG_polygrowth}
	Let $u \in G_+$ be an order unit (\Cref{nicely_ordered}).
	Then $v \coloneqq \delta_u$ is power universal in $\measc{G}$.
\end{lem}

\begin{proof}
	If $\mu,\nu \in \measc{G}$ satisfy $\mu \le \nu$, then in particular $\mu(G) = \nu(G)$. We can thus assume that both $\mu$ and $\nu$ are probability measures without loss of generality. With $\mu_-$ denoting the pushforward of $\mu$ along the inversion map $G \to G$, we first show that there is $k \in \N$ with
	\[
		\nu \le \delta_{ku}, \qquad \mu_- \le \delta_{ku}.
	\]
	Since $\supp{\nu}$ is compact by assumption, the first inequality follows from \Cref{compact_bounded} and \Cref{supp_crit}.
	The argument for $\mu_- \le \delta_{ku}$ is the same.

	We then get
	\[
		\nu \le \delta_{ku} = \delta_{2ku} \ast \delta_{-ku} \le \delta_{2ku} \ast \mu = (\delta_u)^{2k} \ast \mu,
	\]
	which is enough.
\end{proof}

The following finite approximation result will be a crucial stepping stone in the proof of our main result presented in the next section.

\begin{lem}
	\label{finite_support_dense}
	For $u \in G$ an order unit and for every $\mu \in \measc{G}$, the order interval
	\[
		\left[ \mu \ast \delta_{-2u}, \mu \ast \delta_{+2u} \right]
	\]
	contains a finitely supported measure.
\end{lem}

\begin{proof}
	We assume $\mu(G) = 1$ without loss of generality. Since $\supp{\mu}$ is compact by assumption and $[x-u, x+u]$ is a neighbourhood of $x$ for every $x \in G$, we have finitely many $x_1,\ldots,x_n \in \supp{\mu}$ such that
	\[
		\supp{\mu} \subseteq \bigcup_{i=1}^n \: [x_i - u, x_i + u],
	\]
	as in the proof of \Cref{compact_bounded}.
	The sets $\supp{\mu} \cap [x_i - u, x_i + u]$ generate a finite Boolean algebra of measurable sets with atoms $B_1, \ldots, B_m \subseteq \supp{\mu}$. Upon choosing arbitrary points $y_j \in B_j$, we define
	\[
		\nu \coloneqq \sum_{j=1}^m \mu(B_j) \, \delta_{y_j}.
	\]
	We then argue that $\nu \le \mu \ast \delta_{2u}$; the other claimed inequality works analogously. Indeed consider the measure on $G \times G$ given by
	\[
		\lambda \coloneqq \sum_{j=1}^m \delta_{y_j} \otimes (\mu|_{B_j} \ast \delta_{2u}).
	\]
	Its two marginals are $\nu$ and $\mu \ast \delta_{2u}$, respectively, so it is enough to prove that $\lambda$ is supported on the relation $\le$. But this is because of $B_j \subseteq [x_i - u, x_i + u]$ for some $i$, which implies the relevant $y_j \in \down{(B_j + 2u)}$.
\end{proof}

Before we can state our recent \emph{Vergleichsstellensatz}~\cite[Theorem~8.6]{vssII} for preordered semirings, we need one more definition.

\begin{defn}
	If $S$ and $T$ are semialgebras, then a \newterm{semialgebra homomorphism} from $S$ to $T$ is a map $\phi : S \to T$ which preserves addition, multiplication, and scalar multiplication: for all $x,y \in S$ and $r \in \R_+$,
	\[
		\phi(x + y) = \phi(x) + \phi(y), \qquad \phi(xy) = \phi(x) \phi(y), \qquad \phi(rx) = r \phi(x),
	\]
	and also preserves the neutral elements, $\phi(0) = 0$ and $\phi(1) = 1$.
\end{defn}

Now here is~\cite[Theorem~8.6]{vssII}, specialized to the case of preordered semialgebras and to the case where the power universal element $v$ is invertible; both of these assumptions result in some small simplifications.

\begin{thm}
	\label{vss}
	Let $S$ be a preordered semialgebra with a power universal element $v \in S$ that is multiplicatively invertible, and suppose that $S$ comes equipped with a surjective homomorphism $\|\cdot\| : S \to \R_+$ with trivial kernel and such that
	\[
		a \le b \quad \Longrightarrow \quad \|a\| = \|b\| \quad \Longrightarrow \quad a \sim b,
	\]
	where $\sim$ denotes the equivalence relation generated by $\le$.
	\medskip

	\noindent For $x, y \in S$ with $\|x\| = \|y\| = 1$, consider the following conditions:
	\begin{enumerate}
		\item\label{cat_order} There is $a \in S$ with $\|a\| = 1$ and
			\begin{equation}
				a x \le a y.
			\end{equation}
		\item\label{asymp_order} For $n \ge 1$, we have
			\begin{equation}
				\label{asymp_order_eq}
				x^n \le y^n.
			\end{equation}
		\item\label{sper_ineqs} With $\prec$ standing for $<$ or $\le$, the following hold:
			\begin{enumerate}
				\item\label{phi_ineq} For every monotone semialgebra homomorphism $\phi : S \to \mathbb{K}$ with $\mathbb{K} \in \{\R_+, \R_+^\op, \TR_+, \TR_+^\op\}$, and such that $\phi$ has trivial kernel and does not factor through $\|\cdot\|$, we have
					\[
						\phi(x) < \phi(y).
					\]
				\item\label{D_ineq} For every nonzero $\R_+$-linear monotone map $D : S \to \R$ which satisfies the Leibniz rule $D(ab) = D(a) \, \|b\| + \|a\| \, D(b)$, we have
					\[
						D(x) < D(y).
					\]
			\end{enumerate}
	\end{enumerate}
	Then \ref{cat_order} or \ref{asymp_order} for some $n \ge 1$ implies that \ref{sper_ineqs} holds with non-strict inequalities.
	Conversely if \ref{sper_ineqs} holds with strict inequalities, then \ref{cat_order} and \ref{asymp_order} for all $n \gg 1$ follow.
\end{thm}

Here, the forward direction is very simple to prove, and follows directly upon applying the relevant maps $\phi$ and $D$ to the assumed inequality and cancelling either the resulting term involving $a$ or the $n$-th power. The final sentence going from \ref{sper_ineqs} to \ref{cat_order}$+$\ref{asymp_order} is much deeper and requires a substantial theory development~\cite{vssI,vssII}. Note that it is a converse to the forward direction in generic cases, in the sense that the only difference\footnote{Another subtle difference concerns the $n$ in \eqref{asymp_order_eq}. The converse direction is \emph{stronger} in this respect than would be required for a converse, since it allows us to conclude \eqref{asymp_order_eq} \emph{for all} $n \ge 1$ rather than merely for some $n \ge 1$.} is the strictness of the inequalities and inequalities are generically strict.

\begin{rem}
	\label{normalize_general}
	For a monotone homomorphism $\phi : S \to \TR_+$ or $\phi : S \to \TR_+^\op$, also any power $\phi^r$ for $r > 0$ is a homomorphism of the same type.
	This implies that in \ref{phi_ineq}, we can additionally restrict to those $\phi$ which satisfy the normalization condition $\phi(v) = e$ for $\phi : S \to \TR_+$, or $\phi(v) = e^{-1}$ for $\phi : S \to \TR_+^\op$.
	Note that any other real number $> 1$ could be used in place of $e$, but this choice is convenient in combination with the natural logarithm as we use it in \eqref{lev}.

	A similar statement applies to the maps $D$ in~\ref{D_ineq}, where we can restrict to those $D$ which satisfy $D(v) = 1$.
	The reason is that we have $D(v) \ge D(1) = 0$ by monotonicity, and $D(v) = 0$ would imply $D = 0$, which is assumed not to be the case.
\end{rem}

In order to apply \Cref{vss} in a concrete case, it is necessary to characterize first the inequalities required by conditions~\ref{phi_ineq} and~\ref{D_ineq}.
We will do this for $\measc{G}$ in the next section.

Since the inequalities in conditions~\ref{phi_ineq} and~\ref{D_ineq} can seem a bit unwieldy, we now explain how these inequalities can be organized and subsumed into one single structure. 
There are five types of relevant inequalities corresponding to various types of monotone maps out of $S$. In~\cite{vssII}, we have introduced the following terminology for talking about these types.
\begin{itemize}
	\item Monotone homomorphisms $\phi : S \to \TR_+$ are \newterm{max-tropical}.
	\item Monotone homomorphisms $\phi : S \to \R_+$ are \newterm{max-temperate}.
	\item Monotone $\R_+$-linear maps $D : S \to \R$ satisfying the Leibniz rule are \newterm{arctic}.
	\item Monotone homomorphisms $\phi : S \to \R_+^\op$ are \newterm{min-temperate}.
	\item Monotone homomorphisms $\phi : S \to \TR_+^\op$ are \newterm{min-tropical}.
\end{itemize}
It is useful to consider all of these together as defining a family of inequalities parametrized by a suitable topological space.
We recall the construction of this space here, assuming that $S$ is as in \Cref{vss}.

\begin{defn}[{\cite[Section~8]{vssII}}]
	\label{tsper_defn}
	The \newterm{test spectrum} $\Sper{S}$ is the disjoint union
	\begin{align*}
		\Sper{S} \: \coloneqq \:	& \: \{\text{\normalfont{monotone homs }} \phi : S \to \R_+ \text{\normalfont{ or }} \phi : S \to \R_+^\op \} \setminus \{\|\cdot\|\} \\
					& \sqcup \{\text{\normalfont{monotone homs }} \phi : S \to \TR_+ \text{\normalfont{ with }} \phi(v) = e\} \\
					& \sqcup \{\text{\normalfont{monotone homs }} \phi : S \to \TR_+^\op \text{\normalfont{ with }} \phi(v) = e^{-1} \} \\
					& \sqcup \{\text{\normalfont{monotone $\R_+$-linear derivations }} D : S \to \R \text{\normalfont{ with }} D(v) = 1 \}.
	\end{align*}
	and carries the coarsest topology which makes the \newterm{logarithmic evaluation maps}
	\begin{equation}
		\label{lev}
		\lev_x(\phi) \coloneqq \frac{\log \phi(x)}{\log \phi(u)},	\qquad		\lev_x(D) \coloneqq D(x)
	\end{equation}
	continuous for all $x \in S$ with $\|x\| = 1$.
\end{defn}

Note that there is a minor difference relative to~\cite[Definitions~8.3 and 8.4]{vssII}, namely that we now define logarithmic evaluation maps $\lev_x$ rather than logarithmic comparison maps.
Let us explain how this simplification is possible due to the semialgebra structure and results in an equivalent topology.
The logarithmic comparison maps which in terms of the above maps are defined for $x, y \in S$ with $\|x\| = \|y\| = 1$ as
\[
	\lc_{x,y} \coloneqq \lev_y - \lev_x.
\]
It is sufficient to consider the case $\|x\| = \|y\| = 1$ rather than the weaker $\|x\| = \|y\|$ as in \cite[Definition~8.4]{vssII}, since we are only dealing with the semialgebra case, and $x$ and $y$ can be normalized by scalar multiplication.
Then $\lc_{x,y}$ is continuous if both $\lev_x$ and $\lev_y$ are; and conversely, $\lev_x = \lc_{v,x} - 1$ is continuous as soon as the $\lc_{x,y}$ are.

\begin{prop}[{\cite[Proposition~8.5]{vssII}}]
	\label{sper_chaus}
	The test spectrum $\Sper{S}$ is a compact Hausdorff space.
\end{prop}

By construction, every nonzero semiring element $x \in S$ defines a continuous map $\lev_x : \Sper{S} \to \R$, or equivalently an element of the algebra of continuous functions $\C(\Sper{S})$. Moreover, the map
\begin{align*}
	\lev \: : \: {} & S \longrightarrow \C(\Sper{S})	\\
		& x \longmapsto \lev_x
\end{align*}
is reminiscent of the Gelfand transform; although it is not a semiring homomorphism, the fact that it maps multiplication to addition is perfectly sufficient for our purposes.
For $\measc{G}$, we will see in \Cref{main_sec} that it corresponds to the formation of the normalized cumulant-generating function.

In terms of the test spectrum, the conditions~\ref{phi_ineq} and~\ref{D_ineq} of \Cref{vss} can now be rephrased as saying that
\[
	\lev_x < \lev_y,
\]
where the strict inequality must be pointwise strict on $\Sper{S}$.
If these inequalities hold, then the catalytic and asymptotic ordering of \ref{catalytic} and \ref{asymptotic} follow; and conversely if the latter hold, then we must have $\lev_x \le \lev_y$ pointwise.

\section{Asymptotic comparison of random walks}
\label{main_sec}

In order to apply \Cref{vss} to $\measc{G}$, we thus still need to determine the relevant test spectrum. Throughout, $G$ will still be a preordered topological abelian group with positive cone $G_+$ and order unit $u \in G_+$ in the sense of \Cref{nicely_ordered}. We write
\[
	G^*_+ \coloneqq \{ \textrm{monotone group homomorphisms } G \to \R\},
\]
and we also denote the application of $t \in G^*_+$ to a group element $x \in G$ by $\langle t, x\rangle \coloneqq t(x)$.

\begin{lem}
	\label{state_cont}
	Every $t \in G^*_+$ is continuous as a map $t : G \to \R$.
\end{lem}

Although this can be seen as a consequence of the purely algebraic~\cite[Proposition~7.18]{goodearl}, it is easy enough to give a direct proof.

\begin{proof}
	By monotonicity, we obviously have $\langle t, u \rangle \ge 0$.
	If $\langle t, u \rangle = 0$, then $t = 0$ by the order unit property, and hence $t$ is trivially continuous.
	We can therefore assume $\langle t, u \rangle = 1$ without loss of generality.

	It is enough to show that $t$ is continuous at zero, or equivalently that for every $n \in \Nplus$ the set $t^{-1}([-\frac{1}{n},+\frac{1}{n}])$ is a neighbourhood of zero.
	But this is the case because the map
	\[
		G \longrightarrow G, \qquad x \longmapsto n x
	\]
	is continuous, and hence the left-hand side of
	\[
		\{ x \in G \mid -u \le n x \le u \} \: \subseteq \: t^{-1}\Big(\Big[-\frac{1}{n},+\frac{1}{n}\Big]\Big)
	\]
	is a neighbourhood of zero (since it is for $n = 1$ by \Cref{nicely_ordered}).
\end{proof}

We now characterize the five kinds of points of the test spectrum, starting with the max-temperate case.
This involves the moment-generating function.

\begin{lem}
	\label{mgf_char}
	The monotone semialgebra homomorphisms $\phi : \measc{G} \to \R_+$ are precisely the maps of the form
	\beq
		\label{mgf_def}
		\mu \longmapsto \int e^{\langle t, x\rangle} \, d\mu(x)
	\eeq
	for some $t \in G^*_+$. 
\end{lem}

\begin{proof}
	Every such map is clearly $\R_+$-linear by linearity of the integral. The multiplicativity follows by the formula~\eqref{integral_convolute} for the integral of a function against a convolution,
	\[
		\int e^{\langle t, x \rangle} \, d(\mu \ast \nu)(x) = \iint e^{\langle t, y + z \rangle} \, d\mu(y) \, d\nu(z) = \left( \int e^{\langle t, y\rangle} \, d\mu(y) \right) \left( \int e^{\langle t, z\rangle} \, d\nu(z) \right) .
	\]
	Monotonicity in $\mu$ holds because the integrand $x \mapsto e^{\langle t, x\rangle}$ is a monotone and lower semi-continuous function, where the latter is a consequence of \Cref{state_cont}, and therefore \Cref{stochastic_order_char} applies.

	For the converse, let $\phi : \measc{G} \to \R_+$ be a monotone homomorphism. Restricting $\phi$ along the inclusion homomorphism $G \to \measc{G}$ from~\eqref{unit_monad} shows that\footnote{We have $\phi(\delta_x) > 0$ since $\delta_x$ is invertible by $\delta_x \ast \delta_{-x} = \delta_0 = 1$.}
	\[
		t(x) \coloneqq \log \phi(\delta_x)
	\]
	defines an element $t \in G^*_+$. The formula~\eqref{mgf_def} then holds by definition for all delta measures, and $\R_+$-linearity implies that it therefore also holds for all finitely supported measures. But then together with monotonicity, \Cref{finite_support_dense} shows that the value $\phi(\mu)$ for any $\mu$ differs from \eqref{mgf_def} by a factor of at most $\phi(2u)$. Applying this statement to a power $\mu^{\ast n}$ and taking $n \to \infty$ proves that $\phi(\mu)$ actually coincides with \eqref{mgf_def}.
\end{proof}

Let us consider the max-tropical case next.

\begin{lem}
	\label{tr_char}
	The monotone semialgebra homomorphisms $\phi : \measc{G} \to \TR_+$ are precisely the maps of the form
	\beq
		\label{max_defn}
		\mu \longmapsto \exp\left({\max_{x \,\in\, \supp{\mu}} \langle t, x \rangle}\right)
	\eeq
	for some $t \in G^*_+$.
\end{lem}

Since $\supp{\mu}$ is compact and $\langle t, - \rangle$ is continuous by \Cref{state_cont}, the maximum in~\eqref{max_defn} is attained.

\begin{proof}
	Since the support of the sum of two measures is exactly the union of the supports, such a map indeed takes addition to max.
	It also takes the zero measure to $e^{-\infty} = 0$, and a positive scalar multiple $r \mu$ to the same number as any measure $\mu$ itself. Since the support of a convolution is exactly the Minkowski sum of the supports, which becomes multiplication upon exponentiation, it follows that~\eqref{max_defn} indeed defines a semialgebra homomorphism. Monotonicity follows from the characterization of the stochastic preorder of \Cref{stochastic_order_char}\ref{joint} in terms of a joint distribution $\lambda$ supported on the relation $\le$: considering the support of $\lambda$ as a subset of $G \times G$ shows that if $\mu \le \nu$ and $x \in \supp{\mu}$, then there must be $y \in \supp{\nu}$ with $x \le y$.

	Conversely, suppose that $\phi : \measc{G} \to \TR_+$ is a monotone semialgebra homomorphism. Then define an element $t \in G^*_+$ by $t(x) \coloneqq \log \phi(\delta_x)$. The assumption that $\phi$ preserves multiplication shows that $t(x + y) = t(x) + t(y)$, while monotonicity of $t$ is obvious. Hence indeed $t \in G^*_+$. The $\R_+$-linearity of $\phi$ then implies that $\phi$ coincides with \eqref{max_defn} for all finitely supported $\mu$. For general $\mu$, we again use \Cref{finite_support_dense}, which shows that $\phi(\mu)$ differs from the value of \eqref{max_defn} by a factor of at most $\phi(2u)$. Applying this statement to the powers $\mu^{\ast n}$ and taking $n \to \infty$ proves the claim.
\end{proof}

\begin{rem}
	\label{op_char}
	Replacing $G_+$ by $-G_+$ in \Cref{mgf_char} and \Cref{tr_char} shows that the monotone homomorphisms $\measc{G} \to \R_+^\op$ and $\measc{G} \to \TR_+^\op$ are also of the specified form, but with $-t$ in place of $t$. The previous two lemmas thus also characterize the min-temperate and min-tropical parts of the real spectrum.
\end{rem}

\begin{lem}
	\label{der_char}
	The maps $D : \measc{G} \to \R$ which are monotone, $\R_+$-linear and satisfy the Leibniz rule
	\[
		D(\mu\ast\nu) = D(\mu) \nu(G) + \mu(G) D(\nu)
	\]
	are precisely the maps of the form
	\beq
		\label{der_defn}
		\mu \longmapsto \int \langle t,x \rangle \, d\mu(x)
	\eeq
	for $t \in G^*_+$.
\end{lem}

\begin{proof}
	The $\R_+$-linearity of such a map is obvious, and monotonicity holds by monotonicity and continuity of $t$ itself. The Leibniz rule follows again by~\eqref{integral_convolute} and the additivity of $t$,
	\begin{align*}
		\int \langle t, x \rangle \, d(\mu \ast \nu)(x)	& = \iint \langle t, y + z \rangle \, d\mu(y) \, d\nu(z) \\
						& = \left( \int \langle t, y \rangle \, d\mu(y) \right) \nu(G) + \mu(G) \left( \int \langle t, z \rangle \, d\nu(z) \right).
	\end{align*}
	Conversely, suppose that $D : \measc{G} \to \R$ has the relevant properties, and put $t(x) \coloneqq D(\delta_x)$. Then the assumed monotonicity of $D$ and the Leibniz rule show that $t \in G^*_+$. It follows then by $\R_+$-linearity that $D$ coincides with \eqref{der_defn} on the finitely supported $\mu$. To show this for all $\mu$, we consider $\mu(G) = 1$ without loss of generality and again apply \Cref{finite_support_dense}. Together with the Leibniz rule, this shows that $D(\mu)$ for arbitrary $\mu$ differs from \eqref{der_defn} by at most $D(2u)$. Applying this statement to a power $\mu^{\ast n}$ then proves the claim.
\end{proof}

\newcommand{\so}{\,\overset{d}{\le}\,}

We can therefore instantiate \Cref{vss} to the following main result, which we formulate directly in probabilistic terms using random variables. In the following, all inequalities between $G$-valued random variables refer to the preorder on $G$ induced by the positive cone $G_+$, and are to be interpreted as holding almost surely. Let us also say that a random variable is \newterm{Radon} if its distribution is a Radon measure.
For $G$-valued random variables $X$ and $Y$, we also write
\[
	X \so Y
\]
to denote the stochastic preorder in the sense of \Cref{stochastic_order_char} between their distributions.

\begin{thm}
	\label{main_thm}
	Let $G$ be a topological abelian group, preordered with respect to a positive cone $G_+ \subseteq G$ having an order unit $u \in G_+$ such that the order interval $[-u,+u]$ is a neighbourhood of zero. Let all random variables be $G$-valued, compactly supported and Radon, 
	
	Consider the following conditions on random variables $X$ and $Y$:
	\begin{enumerate}[label=(\roman*)]
		\item\label{catalytic} There is a third random variable $Z$, independent of $X$ and $Y$, such that
			\beq
				\label{catalytic_eq}
				X + Z \so Y + Z.
			\eeq
		\item\label{asymptotic} For i.i.d.~copies $(X_i)_{i \in \N}$ and $(Y_i)_{i \in \N}$, there is $n \ge 1$ such that
			\beq
				\label{asymptotic_eq}
				\sum_{i=1}^n X_i \so \sum_{i=1}^n Y_i.
			\eeq
		\item\label{ineqs} With $\prec$ standing for $<$ or $\le$, the following hold for all nonzero $t \in G^*_+$:
			\begin{align}
				\label{mgf_compare}
				\E{e^{\langle t,X \rangle}} & \prec \E{e^{\langle t,Y\rangle}},	& \E{e^{-\langle t,X\rangle}} & \succ \E{e^{-\langle t,Y\rangle}}, \\[6pt]
				\label{tropical_compare}
				\max \langle t,X\rangle & \prec \max \langle t,Y\rangle,		& \min \langle t, X\rangle & \prec \min \langle t, Y\rangle, 
			\end{align}
			\beq
				\label{E_compare}
				\E{\langle t,X\rangle} \prec \E{\langle t,Y\rangle}.
			\eeq
	\end{enumerate}
	Then \ref{catalytic} or \ref{asymptotic} for some $n \ge 1$ implies that \ref{ineqs} holds with non-strict inequalities. Conversely if \ref{ineqs} holds with strict inequalities, then \ref{catalytic} and \ref{asymptotic} for all $n \gg 1$ follow.
\end{thm}

As with \Cref{vss} in general, the forward direction is easy to see by applying the respective functions to the assumed inequality. Our main result is the converse direction, and we state the forward direction mainly to indicate that our converse is generically necessary and sufficient: the only difference is in the strictness of the inequalities in~\ref{ineqs}, and these are strict in generic cases.

\begin{proof}
	This follows upon instantiating \Cref{vss} on $\measc{G}$, taking $\|\cdot\| : \measc{G} \to \R_+$ to be given by the normalization homomorphism $\mu \mapsto \mu(G)$, and translating the statement into random variables language. The power universality of $\delta_u$ holds by \Cref{MG_polygrowth}. The relevant monotone quantities are exactly the specified ones, as per \Cref{mgf_char,tr_char,der_char,op_char}, where we have in addition taken the logarithm of those of the form~\eqref{max_defn} for simplicity.
\end{proof}

For $G = \R$ and $G_+ = \R_+$, \Cref{main_thm} specializes to \Cref{Rintro_thm}, since in this case, the stochastic preorder is characterized by the given inequalities between cumulative distribution functions per \eqref{R_stoch_order}.
Also, it is clear that the inequalities \eqref{tropical_compare} and~\eqref{E_compare} only need to be considered for $t = 1$ then, as we have done in \Cref{Rintro_thm}.

\begin{rem}
	It may be worth pointing out that the question whether a ``catalyst'' $Z$ as in~\eqref{catalytic_eq} and the earlier~\eqref{catalytic_ineq} exists depends strongly on whether $Z$ is required to be compactly supported or not. While our result is concerned with the compactly supported case, another recent result of Pomatto, Strack and Tamuz for $G = \R$ shows that such a $Z$ with merely finite first moment exists already as soon as only $\E{X} < \E{Y}$ holds~\cite{PST}.
\end{rem}

\section{The normalized cumulant-generating function}
\label{ncgf}

For successful applications of \Cref{main_thm}, it is imperative to understand the inequalities~\eqref{mgf_compare}--\eqref{E_compare} well, and in particular how they relate to each other. 
This is what we do in this section, by investigating the structure of the test spectrum of $\measc{G}$. In our current context, this space behaves a lot like a projective version of $G^*_+$, so we denote it by $\Proj{G_+}$.

\begin{defn}
	Let $G$ be a topological abelian group preordered with respect to a positive cone $G_+ \subseteq G$ having an order unit $u \in G_+$. Then $\Proj{G_+}$ is the disjoint union of the following five parts:
	\begin{itemize}
		\item The \newterm{max-tropical part}, given by
			\[
				\{ t \in G^*_+ \mid \langle t, u \rangle = 1 \}.
			\]
		\item The \newterm{max-temperate part}, given by
			\[
				G^*_+ \setminus \{0\}.
			\]
		\item The \newterm{arctic part}, given by
			\[
				\{ t \in G^*_+ \mid \langle t, u \rangle = 1 \}.
			\]
		\item The \newterm{min-temperate part}, given by
			\[
				G^*_+ \setminus \{0\}.
			\]
		\item The \newterm{min-tropical part}, given by
			\[
				\{ t \in G^*_+ \mid \langle t, u \rangle = 1 \}.
			\]
	\end{itemize}
	$\Proj{G_+}$ carries the coarsest topology which makes the \newterm{logarithmic evaluation maps} $\lev_\mu : \Proj{G_+} \to \R$ defined as
	\[
		\lev_\mu(t) \coloneqq \begin{cases}
			\max_{x \,\in\, \supp{\mu}} \langle t, x\rangle		& \textrm{ if } t \textrm{ is max-tropical}, \\[4pt]
			\langle t, u\rangle^{-1} \log \int_G e^{\langle t, x\rangle} \, d\mu(x)	& \textrm{ if } t \textrm{ is max-temperate}, \\[4pt]
			\int_G \langle t, x\rangle \, d\mu(x)		& \textrm{ if } t \textrm{ is arctic}, \\[4pt]
			-\langle t, u\rangle^{-1} \log \int_G e^{-\langle t, x\rangle} \, d\mu(x)& \textrm{ if } t \textrm{ is max-temperate}, \\[4pt]
			\min_{x \,\in\, \supp{\mu}} \langle t, x\rangle		& \textrm{ if } t \textrm{ is max-tropical} \\[4pt]
		\end{cases}
	\]
	continuous for all $\mu \in \measc{G}$ with $\mu(G) = 1$.
\end{defn}

\begin{rem}
	Some comments are in order to make sense of this definition.
	\begin{enumerate}
		\item Each of the five parts of $\Proj{G_+}$ is a copy of the positive cone $G^*_+$, where in the arctic and the two tropical parts, one additionally has a normalization condition (as per \Cref{normalize_general}).
			We will not introduce separate notation for the five parts but distinguish them in words.
		\item Using \Cref{mgf_char}--\ref{der_char} together with some calculation, it is straightforward to see that the thus defined $\Proj{G_+}$ is exactly the test spectrum $\Sper{\measc{G}}$ in the sense of \Cref{tsper_defn}.
			
			In particular, $\Proj{G_+}$ is a compact Hausdorff space by \Cref{sper_chaus}.
		\item For nonzero $t \in G^*_+$, consider the associated max-temperate point of $\Proj{G_+}$, and let $r \in \Rplus$ be a scalar. Then $rt$ again represents a max-temperate point with logarithmic evaluation map given by\footnote{As an interesting aside, this function of $r$ is constant whenever $\mu = \delta_x$ for $x \in G$.}
			\beq
				\label{rlev}
				\mu \longmapsto \frac{\log \int_G e^{r \langle t, x\rangle} \, d\mu(x)}{r \langle t, u\rangle}.
			\eeq
			Assuming that $t$ is normalized to $\langle t, u\rangle = 1$, it also defines a max-tropical and an arctic point of $\Proj{G_+}$. And indeed the corresponding logarithmic evaluation maps arise from \eqref{rlev} as limits in $r$: taking $r \to \infty$ recovers the tropical case,
			\[
				\lim_{r \to \infty} \frac{\log \int_G e^{r \langle t, x\rangle} \, d\mu(x)}{r \langle t, u \rangle} = \max_{x \,\in\, \supp{\mu}} \langle t, x\rangle,
			\]
			and $r \to 0$ recovers the arctic case,
			\[
				\lim_{r \to 0} \frac{\log \int_G e^{r \langle t, x\rangle} \, d\mu(x)}{r \langle t, u \rangle} = \int_G \langle t, x\rangle \, d\mu(x),
			\]
			both of which follow by an elementary calculation, assuming that $\mu$ is a probability measure.
			Therefore in the topology on $\Proj{G_+}$, the arctic and max-tropical parts are path-connected to the max-temperate part.
			
			Analogous statements apply in the min-temperate case, where we get
			\[
				\lim_{r \to \infty} \left(- \frac{\log \int_G e^{-r \langle t, x\rangle} \, d\mu(x)}{r \langle t, u \rangle} \right) = \min_{x \,\in\, \supp{\mu}} \langle t, x\rangle,
			\]
			and
			\[
				\lim_{r \to 0} \left(- \frac{\log \int_G e^{-r \langle t, x\rangle} \, d\mu(x)}{r \langle t, u \rangle} \right) = \int_G \langle t, x\rangle \, d\mu(x).
			\]
			Therefore the arctic and min-tropical parts are path-connected to the min-temperate part.
	\end{enumerate}
\end{rem}

\begin{ex}
	With $G = \R^2$ and $G_+ = \R_+^2$ the positive quadrant, we can construct $\Proj{\R_+^2}$ explicitly as follows.
	We start with the positive quadrant $\R_+^2$ and the negative quadrant $\R_-^2$, corresponding to the max-temperate and min-temperate parts, respectively (discounting the origin).
	We compactify this space by throwing in two additional arcs at infinity, corresponding to the max-tropical and min-tropical parts.
	Finally, we blow up the origin to a line, and this represents the arctic part.
	Overall, we thus obtain a space homeomorphic to the closed unit disk.
\end{ex}

The main use of $\Proj{G_+}$ is that it allows us to summarize the inequalities~\eqref{mgf_compare}--\eqref{E_compare} concisely as a pointwise inequality between continuous functions on $\Proj{G_+}$, namely as
\[
	\lev_X \prec \lev_Y,
\]
where we omit notational distinction between a random variable and its distribution. In the non-strict case, this inequality is simply pointwise inequality $\le$ between continuous functions on $\Proj{G_+}$, and similarly in the strict case. The compactness of this space can be useful in concrete applications of \Cref{main_thm}. We will develop one such application in the next section.

\begin{ex}
	\label{R_ncgf}
	Consider $G = \R$ and $G_+ = \R_+$ with order unit $u = 1$. 
	Then unfolding the definition of the logarithmic evaluation maps shows that these are simply the functions of the form
	\begin{equation}
		\label{lev_ncgf}
		\lev_X(t) = \frac{\log \E{e^{tX}}}{t}
	\end{equation}
	for $t \in \R \setminus \{0\}$, reproducing the correct limits as $t \to 0$ and $t \to \pm \infty$, namely $\E{X}$ and $\min X$ and $\max X$.
	In order to get this simple representation, we have reparametrized $t$ as $-t$ in the min-temperate part, and the five parts correspond then exactly to the extended real line as the disjoint union
	\[
		\overline{\R} = \{-\infty\} \cup (-\infty, 0] \cup \{0\} \cup [0, +\infty) \cup \{+\infty\}.
	\]
	The logarithmic evaluation maps in the form~\eqref{lev_ncgf} are also known as the \newterm{normalized cumulant-generating function}.
	The only difference with respect to the usual cumulant-generating function $t \mapsto \log \E{e^{tX}}$ is the normalization in the denominator. 
	Campbell~\cite{campbell} seems to have been the first to notice that this normalization is exactly what makes the cumulant-generating function have the relevant limiting values at $t = 0$ and $t \to \pm \infty$.
	The normalization also facilitates thinking of the normalized cumulant-generating function as a family of weighted averages: the value $\lev_X(t)$ is always in the interval $[\min X, \max X]$ and coincides with these values for $t = \pm\infty$; and for larger $t$ the averaging attributes higher weight to the right, while being ``unbiased'' at $t = 0$, for which we get the usual expectation value.
\end{ex}

For a general preordered topological abelian group $G$, we regard our $\lev_X : \Proj{G_+} \to \R$ as the \emph{definition} of the normalized cumulant-generating function of a $G$-valued random variable $X$.

\section{A relative large deviation result and Cram\'er's theorem}
\label{sec_uld}

We now formulate a weaker version of \Cref{main_thm}, one where we relax the properties under consideration such that only the moment-generating function $t \mapsto \E{e^{\langle t, X\rangle}}$ for $t \in G^*_+$ matters, corresponding to the max-temperate part of $\Proj{G_+}$.
This will provide another perspective on our results in the form of a relative large deviation result, where the decay probabilities of two random walks are compared to each other.
It has turned out to be convenient to phrase this result only in the more specific case of topological vector spaces rather than general topological abelian groups.

The advantages of the following result over \Cref{main_thm} are that it is somewhat easier to state, that it is closer to traditional large deviation theory, and that it makes more explicit how our results can be thought of as a duality between the asymptotic behaviour of random walks and the moment-generating or cumulant-generating function.

\begin{thm}
	\label{uniformd_ld}
	Let $V$ be a topological vector space preordered with respect to a convex cone $V_+ \subseteq V$ with $u \in V_+$ such that $[-u,+u]$ is a neighbourhood of zero. Let all random variables be $V$-valued, compactly supported and Radon.
	
	For random variables $X$ and $Y$ and i.i.d.~copies $(X_i)_{i \in \N}$ and $(Y_i)_{i \in \N}$, we have
	\beq
		\label{quotient_mgf}
		\sup_{\eps > 0} \, \uplim_{n \to \infty} \, \sup_C \, \frac{1}{n} \log \frac{\P{\frac{1}{n}\sum_{i=1}^n X_i \in C}}{\P{\frac{1}{n}\sum_{i=1}^n Y_i + \eps u \in C}} = \sup_{t \in V^*_+} \log \frac{\E{e^{\langle t, X\rangle}}}{\E{e^{\langle t, Y\rangle}}},
	\eeq
	where $C$ ranges over all closed upsets in $V$ and the statement holds in two versions, with $\uplim_{n \to \infty}$ standing for $\liminf_{n \to \infty}$ or $\limsup_{n \to \infty}$.
\end{thm}

For $V = \R$ and $V_+ = \R$, this specializes to \Cref{Rintro_uniform_ld} from the introduction.

A few further comments are in order before we get to the proof. The fraction on the left-hand side is understood to be $\infty$ if the denominator vanishes and the numerator does not; we similarly stipulate that $\frac{0}{0} \coloneqq 0$, so that those cases for which both vanish do not contribute to $\sup_C$. The fraction on the left is monotonically nondecreasing as $\eps \to 0$, so that the supremum over $\eps > 0$ is equivalently a limit $\eps \to 0$. Finally, replacing $C$ by $C - \eps u$ shows that the $+\eps u$ term in the denominator on the left-hand side can likewise be replaced by an analogous $-\eps u$ term in the numerator, since
\[
	\sup_C \frac{\P{\frac{1}{n}\sum_{i=1}^n X_i - \eps u \in C}}{\P{\frac{1}{n}\sum_{i=1}^n Y_i \in C}} = \sup_C \frac{\P{\frac{1}{n}\sum_{i=1}^n X_i \in C}}{\P{\frac{1}{n}\sum_{i=1}^n Y_i + \eps u \in C}}.
\]
Let us now turn to the proof, which consists of a reduction to \Cref{main_thm}.

\begin{proof}
	We introduce two auxiliary variables in terms of the given ones, and also depending on additional parameters for which we will choose concrete values below.
	\begin{itemize}
		\item For $p \in (0,1)$ and $k \in \N$, consider the variable $X'$ which coincides with $X$ with probability $p$ and is equal to $-ku$ with probability $1-p$.
		\item For $\eps > 0$, consider $Y' \coloneqq Y + \eps u$.
	\end{itemize}
	We also choose corresponding i.i.d.~copies $(X'_i)_{i \in \N}$ and $(Y'_i)_{i \in \N}$.

	These new variables have normalized cumulant-generating functions taking the form, for $t$ in the max-temperate part of $\Proj{V_+}$,
	\begin{align}
	\begin{split}
		\label{levXY}
		\lev_{X'}(t)	& = \frac{\log \left(p \, \E{e^{\langle t, X\rangle}} + (1-p) e^{-k \langle t, u \rangle} \right)}{\langle t, u\rangle},	\\[4pt]
		\lev_{Y'}(t)	& = \frac{\log \E{e^{\langle t, Y\rangle}}}{\langle t, u\rangle} + \eps,
	\end{split}
	\end{align}
	and likewise with $t \mapsto -t$ in the min-temperate part.

	We first prove the inequality $\ge$ in the claimed equation \eqref{quotient_mgf}, with $\liminf_{n \to \infty}$ in place of $\uplim_{n\to\infty}$. 
	This inequality direction is the easy ``forward'' direction, conceptually analogous to the easy forward directions in \Cref{vss,main_thm}. 
	Since this inequality direction is trivial if the left-hand side is $\infty$, we assume that it is finite.
	Then in the definition of $X'$, choose any $p \in (0,1)$ such that
	\[
		- \log p > \sup_{\eps>0} \, \liminf_{n \to \infty} \, \sup_C \frac{1}{n} \log \frac{\P{\frac{1}{n}\sum_{i=1}^n X_i \in C}}{\P{\frac{1}{n}\sum_{i=1}^n Y_i + \eps u \in C}}.
	\]
	The goal is then to show that $- \log p \ge \log \frac{\E{e^{\langle t, X\rangle}}}{\E{e^{\langle t, Y\rangle}}}$ for any $t \in V^*_+$, or equivalently that $p\,\E{e^{\langle t, X\rangle}} \le \E{e^{\langle t, Y\rangle}}$. Indeed for fixed $\eps > 0$ and $p$ as above, choose $n$ such that the inequality
	\[
		- \log p \ge \sup_C \frac{1}{n} \log \frac{\P{\frac{1}{n}\sum_{i=1}^n X_i \in C}}{\P{\frac{1}{n}\sum_{i=1}^n Y_i + \eps u \in C}}
	\]
	still holds. But then this equivalently means that
	\[
		p^n \, \P{\frac{1}{n} \sum_{i=1}^n X_i \in C} \le \P{\frac{1}{n} \sum_{i=1}^n Y'_i \in C} \qquad \forall C.
	\]
	By choosing $k$ large enough and applying \Cref{compact_bounded}, the distribution of $X'$ will have the property that $\frac{1}{n}\sum_{i=1}^n X'_i$ is below the support of $Y$, and hence also below the support of $\frac{1}{n}\sum_{i=1}^n Y'_i$, given that just \emph{one} of the $X'_i$ is equal to $-ku$. This event is complementary to $X'_i = X_i$ for all $i$, which has probability $p^n$. Therefore also
	\[
		\P{\frac{1}{n} \sum_{i=1}^n X'_i \in C} \le \P{\frac{1}{n} \sum_{i=1}^n Y'_i \in C} \qquad \forall C.
	\]
	Thus $\lev_{X'} \le \lev_{Y'}$ on $\Proj{V_+}$ follows from the easy forward direction of \Cref{main_thm}. But then the desired $p\,\E{e^{\langle t, X\rangle}} \le \E{e^{\langle t, Y\rangle}}$ follows from the above formulas~\eqref{levXY} as $\eps \to 0$.

	We now show the other inequality direction $\le$ with $\limsup_{n \to \infty}$ in place of $\uplim_{n \to \infty}$, which is enough to prove the whole claim. Consider now any value of $p$ with
	\beq
		\label{pbound}
		- \log p > \sup_{t \in V^*_+} \log \frac{\E{e^{\langle t, X\rangle}}}{\E{e^{\langle t, Y\rangle}}},
	\eeq
	or equivalently $p \, \E{e^{\langle t, X\rangle}} < \E{e^{\langle t, Y\rangle}}$ for all $t \in V^*_+$, now assuming without loss of generality that the right-hand side of \eqref{quotient_mgf} is finite. Fix $\eps > 0$. Then for every $t$ in the max-temperate part of $\Proj{V_+}$, the formulas~\eqref{levXY} show that there is $k \gg 1$ such that $\lev_{X'}(t) < \lev_{Y'}(t)$. The same strict inequality holds on the max-tropical part because of $\eps > 0$. For every point of the min-tropical, min-temperate and arctic parts, the $+ku$ component of $X'$ dominates for $k \to \infty$, and therefore we can again find $k \gg 1$ such that $\lev_{X'} < \lev_{Y'}$ at every such point. By compactness of $\Proj{V_+}$, it follows that some fixed $k \gg 1$ works for all of these points. Taking all this together, we have $\lev_{X'} < \lev_{Y'}$ on all of $\Proj{V_+}$ for suitable $k \gg 1$.

	Thus \Cref{main_thm} shows that we have, in terms of i.i.d.~copies: for all $n \gg 1$,
	\[
		\P{\frac{1}{n} \sum_{i=1}^n X'_i \in C} \le \P{\frac{1}{n} \sum_{i=1}^n Y'_i \in C} \qquad \forall C.
	\]
	Since $\sum_{i=1}^n X'_i$ coincides with $\sum_{i=1}^n X_i$ with probability at least $p^n$ and because of $\sum_{i=1}^n Y'_i = \sum_{i=1}^n Y_i + \eps u n$, we obtain that for all $n \gg 1$,
	\[
		p^n \, \P{\frac{1}{n} \sum_{i=1}^n X_i \in C} \le \P{\frac{1}{n} \sum_{i=1}^n Y_i + \eps u \in C} \qquad \forall C,
	\]
	which translates into
	\[
		\sup_C \frac{1}{n} \log \frac{\P{\frac{1}{n}\sum_{i=1}^n X_i \in C}}{\P{\frac{1}{n}\sum_{i=1}^n Y_i + \eps u \in C}} \le - \log p.
	\]
	Letting $p$ approach the bound given in~\eqref{pbound} and noting that $\eps > 0$ was arbitrary proves the inequality $\le$ in the claimed equation~\eqref{quotient_mgf} with $\limsup_{n \to \infty}$, which is enough.
\end{proof}

We end this paper by explaining how \Cref{uniformd_ld} specializes further to a version of Cram\'er's classical large deviation theorem, namely \Cref{cramer} below. Notably, the latter arises by taking one of the two variables in \Cref{uniformd_ld} \emph{to be deterministic}. This justifies thinking of \Cref{uniformd_ld}, and thereby also of the stronger \Cref{main_thm}, intuitively as a result on large deviations of one random walk relative to another.
The purpose of this exercise is not to derive an original result, and our \Cref{cramer} is certainly no improvement over existing versions of Cram\'er's theorem.
Rather, our goal here is to showcase the power of \Cref{uniformd_ld}, and thereby indirectly also of \Cref{main_thm} as our main result.

Throughout the following, $V$ is still a preordered topological vector space as in \Cref{uniformd_ld}.

\begin{defn}
	Let $X$ be a $V$-valued compactly supported Radon random variable. Then its \newterm{rate function} $\Lambda^* : V \to [0,\infty]$ is given by
	\[
		\Lambda^*(c) \coloneqq \sup_{t \,\in\, V^*_+} \left( \langle t, c \rangle - \log \E{e^{\langle t,X\rangle}} \right).
	\]
\end{defn}

Note that this differs from the standard definition of the rate function in (infinite-dimensional) large deviation theory~\cite[(1.10)]{BZ}, where the supremum is taken over the whole dual space $V^*$.
Our rate function $\Lambda^*$ is the more natural quantity in our setting, since we will obtain it directly from \Cref{uniformd_ld}, and it results in the formula~\eqref{cramer_eq} without the need for further formation of an infimum over the set involved.
In the one-dimensional case, the relation between the two versions of the rate function, one with supremum over all $t \in \R$ and the other with supremum over all $t \ge 0$, is well-understood~\cite[Lemma~2.2.5]{DZ}.

\begin{lem}
	\label{rate_cont}
	The rate function $\Lambda^*$ is continuous at every $c \in V$ with $\Lambda^*(c) < \infty$.
\end{lem}

\begin{proof}
	Being a pointwise supremum of linear functions, $\Lambda^*$ is convex. In particular the restricted function
	\[
		\R \longrightarrow [0,\infty], \qquad r \longmapsto \Lambda^*(c + ru)
	\]
	is a one-dimensional convex function and hence continuous at every point at which it is finite. Thus since $\Lambda^*(c) < \infty$ by assumption, for given $\eps > 0$ we in particular have $\delta > 0$ such that
	\[
		| \Lambda^*(c \pm \delta u) - \Lambda^*(c) | < \eps.
	\]
	The claim now follows since $\Lambda^*$ is also monotone (as a supremum of monotone functions) and the order interval $[c - \delta u, c + \delta u]$ is a neighbourhood of $c$.
\end{proof}

\begin{cor}
	\label{cramer}
	Let $V$ be a topological vector space preordered with respect to a closed convex cone $V_+ \subseteq V$ with $u \in V_+$ such that $[-u,+u]$ is a neighbourhood of zero. Let $X$ be a $V$-valued random variable, compactly supported and Radon, and let $(X_i)_{i \in \N}$ be i.i.d.~copies. Then for every $c \in V$,
	\begin{equation}
		\label{cramer_eq}
		\lim_{n \to \infty} \frac{1}{n} \log \P{\frac{1}{n}\sum_{i=1}^n X_i \ge c} = - \Lambda^*(c).
	\end{equation}
\end{cor}

\begin{proof}
	In order to match this up with \Cref{uniformd_ld}, we denote the variables that appear in the statement by $Y$ and $Y_i$ instead.

	Consider the special case of \Cref{uniformd_ld} where $X \coloneqq c$ is constant. Then we have, trivially,
	\[
		\P{\frac{1}{n}\sum_{i=1}^n X_i \in C} = \begin{cases}
					1	& \textrm{ if } c \in C, \\
					0	& \textrm{ if } c \not \in C. \end{cases}
	\]
	Therefore the supremum over $C$ in \eqref{quotient_mgf} is achieved at $C = \up\{c\}$, resulting in
	\[
		\sup_{\eps > 0} \, \uplim_{n \in \N} \, \frac{1}{n} \log \frac{1}{\P{\frac{1}{n}\sum_{i=1}^n Y_i \ge c - \eps u}} = \sup_{t \in V^*_+} \log \frac{e^{\langle t, c\rangle}}{\E{e^{\langle t, Y\rangle}}} = \Lambda^*(c),
	\]
	or equivalently
	\[
		\inf_{\eps > 0} \, \uplim_{n \in \N} \, \frac{1}{n} \log \P{\frac{1}{n}\sum_{i=1}^n Y_i \ge c - \eps u} = - \Lambda^*(c).
	\]
	Monotonicity in $\eps$ now shows that, for every $\eps > 0$,
	\begin{align*}
		- \Lambda^*(c + \eps u) & \le \liminf_{n \to \infty} \frac{1}{n} \log \P{\frac{1}{n} \sum_{i=1}^n Y_i \ge c} \\
					& \le \limsup_{n \to \infty} \frac{1}{n} \log \P{\frac{1}{n} \sum_{i=1}^n Y_i \ge c} \le -\Lambda^*(c).
	\end{align*}
	Thus the claim follows in the limit $\eps \to 0$ by continuity of $\Lambda^*$, \Cref{rate_cont}.
\end{proof}

\begin{rem}
	Let us emphasize again that \Cref{cramer} is merely an illustration of how \Cref{uniformd_ld} can be applied and how the rate function $\Lambda^*$ naturally comes out of it.
	We do not claim any originality for it, and we suspect that it can be recovered as a special case of existing results such as~\cite[Theorem~3.2]{BZ}.
	However, the details have eluded us thus far.\footnote{One difficulty is already that we do not assume the topology on $V$ to be locally convex Hausdorff, although~\cite{BZ} does. One can try to replace the topology on $V$ by the one generated by $\|x\| := \inf \{ r > 0 \mid -r u \le x \le ru \}$. However, this seminorm is not even a norm in general, as one can see e.g.~by considering $\R^2$ with the lexicographic order.}

	Of course, for $V = \R$ and $V_+ = \R_+$, \Cref{cramer} recovers Cram\'er's theorem as stated in many textbooks, such as~\cite[Theorem~23.3]{klenke}, restricted to the case of bounded variables.
\end{rem}

\bibliographystyle{plain}
\bibliography{asymptotic_random_walks}

\end{document}